\documentclass{amsart}
\RequirePackage[utf8]{inputenc}

\usepackage{amsmath,amsthm,amsfonts,amssymb,amscd,amsbsy,dsfont,multirow,hyperref}

\newcommand{\dd}{\mathrm{d}}
\newcommand{\Met}{\mathrm{Met}}
\newcommand{\scal}{\mathrm{scal}}
\newcommand{\Ric}{\mathrm{Ric}}
\newcommand{\R}{\mathds{R}}
\newcommand{\Z}{\mathds{Z}}
\newcommand{\C}{\mathds{C}}
\newcommand{\Hr}{\mathds{H}}

\newcommand{\N}{\mathds{N}}

\newcommand{\vol}{\mathrm{vol}}
\newcommand{\Vol}{\mathrm{Vol}}
\newcommand{\Ad}{\mathrm{Ad}}

\newcommand{\spec}{\mathrm{Spec}}
\newcommand{\grad}{\operatorname{grad}}

\newcommand{\Sp}{\mathrm{Sp}}
\newcommand{\SU}{\mathrm{SU}}
\newcommand{\U}{\mathrm{U}}

\newcommand{\SO}{\mathrm{SO}}
\newcommand{\Spin}{\mathrm{Spin}}

\newcommand{\g}{\mathtt g}
\newcommand{\h}{\mathtt h}
\renewcommand{\k}{\mathtt k}

\hypersetup{
    pdftoolbar=true,
    pdfmenubar=true,
    pdffitwindow=false,
    pdfstartview={FitH},
    pdftitle={Multiplicity of solutions to the Yamabe problem on collapsing Riemannian submersions},
    pdfauthor={Renato G. Bettiol and Paolo Piccione},
    pdfsubject={Primary: 58J55, 53C30; Secondary: 53A30, 53C20, 53C21, 58E50, 58J50},
    pdfkeywords={}
    pdfnewwindow=true,
    colorlinks=true, 
    linkcolor=blue,
    citecolor=blue,
    urlcolor=blue,
}

\newtheorem{theorem}{Theorem}[]

\newtheorem{proposition}[theorem]{Proposition}

\newtheorem{mainthm}{Theorem}

\theoremstyle{definition}

\theoremstyle{remark}
\newtheorem{remark}[theorem]{Remark}
\newtheorem{example}[theorem]{Example}

\title[Multiplicity of solutions to the Yamabe problem on submersions]{Multiplicity of solutions to the Yamabe problem on collapsing Riemannian submersions}
\author{Renato G. Bettiol \and Paolo Piccione}

\allowdisplaybreaks
\numberwithin{equation}{section}
\numberwithin{theorem}{section}

\address{\begin{tabular}{lll}
University of Notre Dame & &Universidade de S\~ao Paulo \\
Department of Mathematics & & Departamento de Matem\'atica \\
255 Hurley Building & & Rua do Mat\~ao, 1010 \\
Notre Dame, IN, 46556-4618, USA & & S\~ao Paulo, SP, 05508-090, Brazil\\
\emph{E-mail address}: {\tt rbettiol@nd.edu} & & \emph{E-mail address}: {\tt piccione@ime.usp.br}
\end{tabular}
}

\date{\today}
\thanks{The second named author is supported by Fapesp, S\~ao Paulo, Brazil, and by CNPq, Brazil. This paper was concluded during a visit of the first named author to the University of S\~ao Paulo, Brazil, in May 2012. We thankfully acknowledge the financial support for this visit from Fapesp, Process 2011/21362-2,  \emph{``Group actions, submanifold theory and global analysis in Riemannian and pseudo-Riemannian geometry''}.}

\subjclass[2010]{Primary: 58J55, 53C30; Secondary: 53A30, 53C20, 53C21, 58E50, 58J50}

\begin{document}
\begin{abstract}
Let $g_t$ be a family of constant scalar curvature metrics on the total space of a Riemannian submersion obtained by shrinking the fibers of an original metric $g$, so that the submersion \emph{collapses} as $t\to0$ (i.e., the total space converges to the base in the Gromov-Hausdorff sense). We prove that, under certain conditions, there are at least $3$ unit volume constant scalar curvature metrics in the conformal class $[g_t]$ for infinitely many $t$'s accumulating at $0$. This holds, e.g., for homogeneous metrics $g_t$ obtained via Cheeger deformation of homogeneous fibrations with fibers of positive scalar curvature.
\end{abstract}

\maketitle
\vspace{-0.5cm}

\section{Introduction}

A classic problem in Riemannian geometry is to find possible \emph{canonical} metrics on a given smooth manifold $M$. Along this quest, an important achievement was the complete solution of the celebrated Yamabe problem, which states that given a closed Riemannian manifold $(M,g_0)$, with $\dim M\geq 3$, there exists a constant scalar curvature metric $g$ conformal to $g_0$. Up to a normalization, such a metric $g$ can be characterized variationally as a critical point of the Hilbert-Einstein functional
\begin{equation}\label{eq:A}
\mathcal A(g)=\frac{1}{\Vol(g)}\int_M \scal(g)\; \vol_g,
\end{equation}
restricted to the set $[g_0]_1$ of unit volume metrics in the conformal class of $g_0$. Existence of a metric that \emph{minimizes} this constrained functional, called the \emph{Yamabe metric}, is a consequence of the works of Yamabe~\cite{Yam60}, Trudinger~\cite{Tru68}, Aubin~\cite{Aub76} and Schoen~\cite{Schoen84}. In addition to the minimizer, there may be also other critical points; thus the solution \emph{may be not unique}. However, Anderson~\cite{And05} recently proved that, on generic conformal classes, the Yamabe metric is the unique solution. In the present paper, we are interested in the complementary situation, i.e., finding conformal classes where the Yamabe problem has \emph{multiple} solutions. Our main results provide a large class of manifolds whose conformal class contains at least $3$ distinct solutions (see Theorems~\ref{thm:main} and \ref{thm:curv}).

A classic method to obtain new solutions of a PDE from a path of known solutions is to use \emph{Bifurcation Theory}. The basic setup for our framework consists of a given one-parameter family $g_t$ of known solutions to the Yamabe problem,
\begin{equation}\label{eq:critpts}
\dd\big(\mathcal A|_{[g_t]_1}\big)(\hat g_t)=0, \quad t\in [a,b],
\end{equation}
where $\hat g_t$ is the unit volume metric homothetic to $g_t$. We study the case where $g_t$ is obtained by shrinking the fibers of a Riemannian submersion with totally geodesic fibers. By proving that certain topological invariants (e.g., the Morse index) of $g_t$ change as $t$ crosses a value $t_*$, one obtains existence of \emph{new} solutions accumulating at $g_{t_*}$. Then, a simple trick (Proposition~\ref{prop:mult}) implies that for $t$'s close to $t_*$ there are at least 3 solutions to the Yamabe problem on $[g_t]_1$.

In our main result, $g_t$ are homogeneous metrics, hence trivially solutions to the Yamabe problem and good (i.e., non-generic) candidates for admitting other solutions in their conformal class. Let $H\subsetneq K\subsetneq G$ be compact connected Lie groups with $\dim K/H\geq2$, and assume that either $H$ is normal in $K$ or $K$ is normal in $G$. Consider the \emph{homogeneous fibration}
\begin{equation}\label{eq:homfib}
K/H\longrightarrow G/H\stackrel{\pi}{\longrightarrow} G/K, \quad \pi(gH)=gK.
\end{equation}
More precisely, $\pi$ is the associated bundle with fiber $K/H$ to the $K$-principal bundle $G\to G/H$. Endow the above spaces with \emph{compatible} homogeneous metrics (see Section \ref{sec:bifhom}). Shrinking the fibers of \eqref{eq:homfib} by a factor $t^2$, we get a family $g_t$ of homogeneous metrics on $G/H$, sometimes called \emph{canonical variation} of \eqref{eq:homfib}. This is, up to reparameterization, the \emph{Cheeger deformation} of $G/H$ in the direction of the natural $K$-action. As $t$ approaches $0$, the manifolds $(G/H,g_t)$ converge to the base $G/K$ in the Gromov-Hausdorff sense. We explore the existence of infinitely many bifurcations of \eqref{eq:critpts} along this collapse of $(G/H,g_t)$ onto the base to obtain the following multiplicity result.

\begin{mainthm}\label{thm:main}
Let $K/H\to G/H\to G/K$ be a homogeneous fibration endowed with compatible homogeneous metrics such that $K/H$ has positive scalar curvature. Let $g_t$ be the family of $G$-invariant metrics on $G/H$ obtained as described above. Then, there exists a subset $\mathcal T\subset\left]0,1\right[$, that accumulates at $0$, such that for each $t\in\mathcal T$ there are at least $2$ solutions to the Yamabe problem in $[g_t]_1$, other than $\hat g_t$, and they are not $G$-invariant.
\end{mainthm}

Theorem~\ref{thm:main} applies, e.g., to short exact sequences $K\to G\to G/K$ of compact connected Lie groups and \emph{twisted product} fibrations $K/H\to ((K\times L)/\Gamma)/H \to G/K$, where $K/H$ is a compact homogenous space of positive scalar curvature and $G=(K\times L)/\Gamma$ is a connected compact Lie group. More interestingly, the result also holds, e.g., for any homogeneous fibrations \eqref{eq:homfib} where $H$ is normal in $K$ and the quotient $K/H$ is a non-abelian compact connected Lie group. In this case, $G$ can have arbitrarily large rank and dimension, and the corresponding possible total spaces can be much more general than twisted products. For more details and examples, see Section~\ref{sec:metrics}. We also stress that the hypotheses made on the scalar curvature (and dimension) of the fibers $K/H$ are necessary. A counter-example is given in the end of Section~\ref{sec:bifhom}.

Various other non-uniqueness phenomena have been studied in the literature, but usually this can only be achieved for very specific examples, see \cite{am,bm,bp,fm,lpz,Schoen91}. One exception is a remarkable result of Pollack~\cite{pollack}, that proved existence of arbitrarily $C^0$-small perturbations of any given metric, with arbitrarily large number of solutions in its conformal class. Previously, Schoen~\cite{Schoen91} had proven existence of an increasing number of solutions, with larger energy and Morse index, in the conformal class of the product $S^1(r)\times S^{m-1}$ of round spheres, as $r$ tends to infinity. Lima, Piccione and Zedda~\cite{lpz} generalized this result to families of product metrics on a product $M_1\times M_2$ of compact Riemannian manifolds given by rescaling one of the factors, obtaining \emph{bifurcation} of solutions. Inspired by these results, the authors recently obtained similar bifurcation results for families of homogeneous spheres in \cite{bp}. The core of Theorem~\ref{thm:main} is a further generalization, establishing that such bifurcations indeed occur on several other families of compact homogeneous spaces.

The initial approach to detect bifurcation in a variational problem, such as \eqref{eq:critpts}, is to look for a change in the Morse index. This is done identifying eigenvalues of the second variation \eqref{eq:jacobi} of $\mathcal A|_{[g_t]_1}$ that change sign for certain values of $t$. Nevertheless, a subtle \emph{compensation} problem may occur, when other eigenvalues with the same combined multiplicity cross zero in the opposite direction. On the one hand, it is not hard to detect passage through zero of certain explicitly computable parts of the spectrum. On the other hand, it is in general not possible to rule out that other eigenvalues also change sign at the same time, potentially producing this compensation that leaves the Morse index unchanged.

In the case of the Yamabe problem, the spectral analysis required consists of comparing eigenvalues of the Laplacian of Riemannian submersions with the scalar curvature. Since the fibers are assumed totally geodesic, eigenfunctions of the Laplacian of the base may be lifted to eigenfunctions of the total space that are constant along the fibers (see Section~\ref{sec:lapl}). This provides a subset of the spectrum easier to deal with, in the sense that a direct computation of the scalar curvature immediately gives infinitely many crossings through zero. Nevertheless, no information is available in general regarding possible compensation due to other crossings in the opposite direction.

The key to handle this situation is to use homogeneity, placing the problem in an equivariant context where a subtler bifurcation criterion due to Smoller and Wasserman~\cite{smwas} applies. Linearizing the action, one gets a representation of the symmetry group in each eigenspace of the second variation. The direct sum of those representations that correspond to negative eigenvalues is called \emph{negative isotropic representation}, see Subsection~\ref{sub:equivbif}. The equivariant criterion asserts that bifurcation occurs at the degeneracy values where the negative isotropic representation changes, which is the case of all degeneracy values mentioned above that correspond to crossings of eigenvalues of the base.
It is then a simple observation that such bifurcations yields the desired multiplicity result (see Proposition~\ref{prop:mult}).

Although homogeneity is strongly used in our main result, one can replace it by other hypotheses that make it possible to avoid the compensation problem. The general context is then a Riemannian submersion with totally geodesic fibers,
\begin{equation}\label{eq:submersion}
F\longrightarrow M\stackrel{\pi}{\longrightarrow} B,
\end{equation}
and metrics $g_t$ on $M$, obtained by shrinking the fibers. Unfortunately, in general, these alternative hypotheses are quite restrictive. Since avoiding this compensation is the central issue in our results, it is natural to expect that a deeper understanding of this issue would allow weaker hypotheses, see Remark~\ref{rem:improvements}. One possibility is to impose curvature conditions that imply certain lower bounds on eigenvalues of the Laplacian, in which case, bifurcation is obtained via the easier Morse index criterion.

\begin{mainthm}\label{thm:curv}
Let $F\to M\to B$ be a Riemannian submersion with totally geodesic fibers and $l=\dim F\geq 2$, $m=\dim M$. Assume the metrics $g_t$ obtained by shrinking the fibers have constant scalar curvature\footnote{For example, this happens if the original metric on $M$ is Einstein, see \cite[Cor 9.62]{besse}.} and that for some $\tau>0$ and $k_1,k_2>0$,
\begin{equation*}
\left\{\begin{matrix}
\Ric_F \!\!\!\!& \geq &\!\!\!\! (l-1)\,k_1 \\
\scal_F \!\!\!\!&<&\!\!\!\! l(m-1)\,k_1
\end{matrix} \right.
\quad\quad\mbox{and}\quad\quad
\left\{\begin{matrix}
\Ric_{(M,g_{\tau})} \!\!\!\!&\geq & \!\!\!\!(m-1)\,k_2 \\
\scal_B \!\!\!\!&\leq &\!\!\!\! m(m-1)\,k_2.
\end{matrix} \right.
\end{equation*}
Then, there exists an infinite subset $\mathcal T$ of positive real numbers, that accumulates at $0$, such that for each $t\in\mathcal T$ there are at least $3$ solutions to the Yamabe problem in the conformal class $[g_t]$.
\end{mainthm}

The paper is organized as follows. In Section~\ref{sec:basics}, we briefly review the variational setup for the Yamabe problem and the bifurcation techniques (Propositions~\ref{prop:bifmorseindex} and \ref{prop:equivbif}) introduced in \cite{bp,lpz}. The effect of shrinking the fibers on the spectrum of a Riemannian submersion with totally geodesic fibers is recalled in Section~\ref{sec:lapl}. The core of the proof of Theorem~\ref{thm:main} (Theorem~\ref{thm:bifhom}) is given in Section~\ref{sec:bifhom}. Section~\ref{sec:metrics} describes several examples to which these theorems apply. Section~\ref{sec:curv} contains the core of the proof of Theorem~\ref{thm:curv} (Theorem~\ref{thm:bifcurv}). Finally, Section~\ref{sec:final} explains how to translate bifurcation of solutions into the multiplicity results claimed above.


\section{Variational framework and Bifurcation criteria}\label{sec:basics}

We start by briefly recalling the classic variational setup for the Yamabe problem, see~\cite{bp,lpz,Schoen87} for details. Let $M$ be a closed smooth manifold of dimension $m$. Consider the set $\Met(M)$ of $C^{r,\alpha}$ Riemannian metrics on $M$, which is an open convex cone in the Banach space of $C^{r,\alpha}$ symmetric $(0,2)$-tensors. Henceforth we fix $r\geq 3$ and $\alpha\in\,]0,1[$. For each $g\in\Met(M)$, define its $C^{r,\alpha}$ conformal class by
\begin{equation*}
[g]=\left\{\phi\,g:\phi\in C^{r,\alpha}(M),\phi>0\right\}.
\end{equation*}
Denote by $\Met_1(M)$ the smooth codimension $1$ embedded submanifold of $\Met(M)$ formed by unit volume metrics. Finally, let
\begin{equation*}
[g]_1=\Met_1(M)\cap [g].
\end{equation*}
The set $[g]_1$ is a codimension $1$ Banach submanifold of $[g]$, and its tangent space at the metric $g$ can be canonically identified as
\begin{equation}\label{eq:idtmet}
T_{g}[g]_1\cong\left\{\psi\in C^{r,\alpha}(M):\int_M\psi\; \vol_{g}=0\right\}.
\end{equation}
The choice of H\"older regularity $C^{r,\alpha}$ for the metric tensors and functions is due to a technical analytic aspect of our theory, that employs certain \emph{Schauder estimates}. To simplify the exposition, these technicalities will not be discussed in further detail.

\subsection{Hilbert-Einsten functional}\label{subs:var}
Fix a metric $g_0$ on $M$ and consider the Hilbert-Einstein functional $\mathcal A$ defined in \eqref{eq:A}. The restriction of $\mathcal A$  to $[g_0]_1$ is smooth, and its critical points are the constant scalar curvature metrics in $[g_0]_1$. Given one such critical point $g\in [g_0]_1$, the second variation of $\mathcal A$ at $g$ can be identified with the quadratic form on \eqref{eq:idtmet} given by
\begin{equation}\label{eq:jacobiop}
\dd^2\big(\mathcal A|_{[g]_1}\big)(g)(\psi,\psi)=\tfrac{m-2}{2}\int_M\big((m-1)\Delta_{g}\psi-\scal(g)\psi\big)\psi\;\vol_{g},
\end{equation}
where $\Delta_g$ is the Laplace-Beltrami operator of $g$, or \emph{Laplacian} of $g$, with the sign convention such that its spectrum is non-negative. The above quadratic form is represented by the (formally) self-adjoint elliptic operator
\begin{equation}\label{eq:jacobi}
J_g(\psi)=\tfrac{(m-1)(m-2)}{2}\left(\Delta_g\psi-\tfrac{\scal(g)}{m-1}\,\psi\right),
\end{equation}
that we call the \emph{Jacobi operator} at $g$. From the above formulas, the critical point $g\in [g_0]_1$ is \emph{nondegenerate} (in the usual sense of Morse theory) if and only if $\frac{\scal(g)}{m-1}$ is not an eigenvalue of the Laplacian $\Delta_g$, or if $\scal(g)=0$. Moreover, the \emph{Morse index} $N(g)$ of this critical point is given by the number of negative eigenvalues of \eqref{eq:jacobi}, i.e., the number of \emph{positive} eigenvalues of $\Delta_g$, counted with multiplicity, that are less than $\tfrac{\scal(g)}{m-1}$. More precisely, denote by
\[0<\lambda_1^g\le\lambda_2^g\le\lambda_3^g\le\ldots\le\lambda_j^g\le\ldots\]
the sequence of eigenvalues of the Laplacian $\Delta_g$, repeated according to their multiplicity. Then, the Morse index of $g$ is given by
\begin{equation}\label{eq:morseidx}
N(g)=\max\left\{j\in\N:\lambda_j^g<\tfrac{\scal(g)}{m-1}\right\}.
\end{equation}

\begin{remark}\label{rem:renorm}
For the purposes of this paper, the relevant data are the signs of the eigenvalues of the operator \eqref{eq:jacobi}. Note that, given $\alpha>0$, one has $\Delta_{\alpha g}=\frac1\alpha\Delta_g$ and $\scal(\alpha g)=\frac1\alpha\scal(g)$. Hence, the spectrum of \eqref{eq:jacobi} scales in a trivial way under homotheties, in the sense that negative (respectively positive) eigenvalues remain negative (respectively positive). On the other hand, $\vol_{\alpha g}=\alpha^\frac m2\vol_g$. Thus, whenever necessary, we may renormalize a metric to have unit volume without compromising the above spectral theory.
\end{remark}

\subsection{A classic result of variational bifurcation theory}
Let us turn to the main tool used in this paper, bifurcation theory. Consider a continuous path $g_t\in\Met(M)$ of solutions to the Yamabe problem, as in \eqref{eq:critpts}. A value $t_*\in [a,b]$ is a \emph{bifurcation value} for the family $g_t$ if there exists a sequence $\{t_q\}$ in $[a,b]$ that converges to $t_*$ and a sequence $\{g_q\}$ in $\Met(M)$ that converges to $g_{t_*}$ satisfying for all $q\in\N$:
\begin{itemize}
\item[(i)] $g_{t_q}\in [g_q]$, but $g_q\neq g_{t_q}$;
\item[(ii)] $\Vol(g_q)=\Vol(g_{t_q})$;
\item[(iii)] $\scal(g_q)$ is constant.
\end{itemize}
If $t_*\in[a,b]$ is \emph{not} a bifurcation value, then the family $g_t$ is \emph{locally rigid} at $t_*$. In other words, the family $g_t$ is locally rigid at $t_*\in[a,b]$ if there exists a neighborhood $U$ of $g_{t_*}$ in $\Met(M)$ such that, for $t\in[a,b]$ sufficiently close to $t_*$, the conformal class $[g_t]$ contains a unique metric of constant scalar curvature in $U$ whose volume equals the volume of $g_t$.

Using a suitable version of the Implicit Function Theorem, one sees that if $g_{t_*}$ is a nondegenerate critical point, then $g_t$ is locally rigid at $t_*$, see~\cite[Prop 3.1]{lpz}. Thus, \emph{degeneracy} is a necessary condition for bifurcation, however it is in general not sufficient. A classic result in variational bifurcation theory states that, given a continuous path of smooth functionals and a continuous path of critical points, there is a bifurcating branch issuing from the given path at each point where the Morse index changes (see \cite[Thm II.7.3]{kielhofer} or \cite[Thm 2.1]{smwas}). Translating this result to our variational framework, we get the following.

\begin{proposition}\label{prop:bifmorseindex}
Let $[a,b]\ni t\mapsto g_t\in\Met(M)$ be a continuous path of constant scalar curvature metrics on $M$, and
assume that $a$ and $b$ are not degeneracy values for $g_t$. If $N(g_a)\neq N(g_b)$, then there exists a bifurcation value $t_*\in\left]a,b\right[$ for the family $g_t$.
\end{proposition}
\begin{proof} See \cite[Thm 3.3]{lpz}.\end{proof}

\subsection{Equivariant variational bifurcation}\label{sub:equivbif}
In many applications, the criterion of Proposition~\ref{prop:bifmorseindex} cannot be employed because establishing a change of the Morse index at a given degeneracy value may be a difficult task. However, when the variational setup has an \emph{equivariant} nature, one can replace the change of Morse index condition with a more general condition based on the representation theory of the group of symmetries of the variational problem. Such more general condition fits perfectly the setup discussed in the present paper, and we now describe it, in the variational framework of the Yamabe problem.

Suppose there exists a finite-dimensional \emph{nice}\footnote{A group $G$ is \emph{nice} if, given unitary representations $\pi_1$ and $\pi_2$ of $G$ on Hilbert spaces $V_1$ and $V_2$ respectively, such that $B_1(V_1)/S_1(V_1)$ and $B_1(V_2)/S_1(V_2)$ have the same (equivariant) homotopy type as $G$-spaces, then $\pi_1$ and $\pi_2$ are unitarily equivalent. Here, $B_1$ and $S_1$ denote respectively the unit ball and the unit sphere in the specified Hilbert space, and the quotient $B_1(V_i)/S_1(V_i)$ is meant in the topological sense (i.e., it denotes the unit ball of $V_i$ with its boundary contracted to one point).

For example, any compact connected Lie group $G$ is nice. More generally, any compact Lie group with less than $5$ connected components is nice. Denoting by $G^0$ the identity connected component of $G$, then $G$ is nice if the discrete part $G/G^0$ is either the product of a finite number of copies of $\mathds Z_2$ (e.g., the case $G=\mathrm O(n)$); or the product of a finite number of copies of $\mathds Z_3$; or if $G/G^0=\mathds Z_4$.}  (in the sense of~\cite{smwas}) Lie group $K$ that acts (on the left) by diffeomorphisms on a compact manifold $M$, and let $[a,b]\ni t\mapsto g_t\in\Met(M)$ be a continuous path of constant scalar curvature metrics on $M$. Up to a suitable normalization, let us assume that each $g_t$ has unit volume (see Remark~\ref{rem:renorm}). The $K$-action on $M$ induces a (right) $K$-action on $\Met(M)$ by pull-back; i.e., the action of $k\in K$ on $g\in\Met(M)$ is $k^*g$. Assume that the $K$-action on $M$ is isometric with respect to all metrics $g_t$, i.e.,
\begin{equation}\label{eq:gtfixed}
k^*(g_t)=g_t, \quad\mbox{ for all } k\in K, \, t\in [a,b].
\end{equation}
In this situation, the $K$-action on $\Met(M)$ leaves every conformal class $[g_t]$ invariant, and also $[g_t]_1$ is $K$-invariant for all $t$. Note that \eqref{eq:gtfixed} means that, for all $t$, $g_t$ is a fixed point of the $K$-action on $\Met(M)$.

It is easy to see that, given $\phi\colon M\to\R$ a positive function,
\begin{equation*}
\mathcal A\big(\phi^{\frac{4}{n-2}}\,g_t\big)=\int_M \left(4\tfrac{n-1}{n-2}\, \phi\,\Delta_{t}\phi+\scal(g_t)\phi^2\right)\vol_{g_t},
\end{equation*}
where $\Delta_t$ is the Laplacian of $g_t$. Using that $k^*(\phi\,g_t)=(\phi\circ k)g_t$, right-composition with isometries commutes with $\Delta_{t}$ and $\scal(g_t)$ is constant, it follows from a change of variables that
\begin{equation*}
\mathcal A\big(k^*(\phi\,g_t)\big)=\mathcal A\big(\phi\,g_t\big), \quad\mbox{ for all } k\in K, \, t\in [a,b].
\end{equation*}
Thus, denoting by $\mathcal A_t$ the restriction of the Hilbert-Einstein functional \eqref{eq:A} to the conformal class $[g_t]_1$, we have that $\mathcal A_t$ is invariant under the $K$-action on $[g_t]_1$.

For each eigenvalue $\lambda\in\spec(\Delta_t)$, denote by $E_t^\lambda\subset L^2(M,\vol_{g_t})$ the correspondent eigenspace. Elements of $E^\lambda_t$ are smooth functions, and $\dim E_t^\lambda<+\infty$ is the multiplicity of $\lambda$. It is easy to see that, for every $t\in [a,b]$ and $\lambda\in\spec(\Delta_t)$, there is a representation $\pi_t^\lambda$ of $K$ in $E_t^\lambda$, given by right-composition with isometries:
\begin{equation}\label{eq:repr}
\pi_t^\lambda(k)\psi=\psi\circ k, \quad\mbox{ for all } k\in K, \, \psi\in E^\lambda_t.
\end{equation}
Note this is (the restriction to $E^\lambda_t$ of) the isotropy representation of the $K$-action on $[g_t]_1$ at the fixed point $g_t$, i.e., the linearization of this $K$-action at $g_t$. We remark that since the $K$-action by pull-back on $[g_t]_1$ is a \emph{right} action, its linearization at a fixed point is actually an \emph{anti}-representation. However, we will henceforth not make distinctions between left/right actions and representations/anti-representations since this does not affect our arguments.

Denote by $\mathcal N_t$ the set of negative eigenvalues of the operator \eqref{eq:jacobi} i.e.,
\begin{equation*}
\mathcal N_t:=\spec\big(J_{g_t}\big)\bigcap\,]-\infty,0[.
\end{equation*}
Now, for all $t\in [a,b]$, define the \emph{negative isotropic representation} $\pi^-_t$ of $K$ in $\bigoplus_{\lambda\in\mathcal N_t}E_t^\lambda$ as the direct sum representation:
\begin{equation*}
\pi^-_t:=\bigoplus_{\lambda\in \mathcal N_t} \pi_t^\lambda.
\end{equation*}
Notice that the degree\footnote{i.e., the dimension of the vector space $\bigoplus_{\lambda\in \mathcal N_t} E_t^\lambda$.} of $\pi^-_t$
is always finite and equal to the Morse index $N(g_t)$ defined in \eqref{eq:morseidx}. Finally, recall that two representations $\pi_i\colon K\to\mathrm{GL}(V_i)$, $i=1,2$, are \emph{isomorphic} if there exists a $K$-equivariant isomorphism $T\colon V_1\to V_2$, i.e., such that $\pi_2(k)\circ T=T\circ\pi_1(k)$ for all $k\in K$.

\begin{proposition}\label{prop:equivbif}
Let $K$ be a nice Lie group acting on a compact manifold $M$ and $[a,b]\ni t\mapsto g_t\in\Met(M)$ be a continuous path of (unit volume) constant scalar curvature metrics. Suppose the $K$-action is isometric on $(M,g_t)$, for all $t\in [a,b]$. If $a$ and $b$ are not degeneracy values for $g_t$, and if the negative isotropic representations $\pi^-_a$ and $\pi^-_b$ are not isomorphic, then there exists a bifurcation value $t_*\in\,]a,b[$ for the family $g_t$.
\end{proposition}

\begin{proof}
This result is a direct application of an equivariant bifurcation result due to Smoller and Wasserman~\cite[Thm~3.1]{smwas}. More precisely, one needs a slightly more general statement of the result, applied to functionals defined on a varying manifold. In our case, the functional
is the Hilbert-Einstein functional $\mathcal A$, defined on the varying manifold $[g_t]_1$. Details can be found in \cite[Thm 3.4, Thm A.2]{lpz}.
\end{proof}

Notice the above is a refinement of Proposition~\ref{prop:bifmorseindex}, which corresponds to saying that $\pi^-_a$ and $\pi^-_b$ do not have the same degree (and hence cannot be isomorphic).


\section{Laplacian on collapsing Riemannian submersions}\label{sec:lapl}

In order to study bifurcation from an initial family $g_t$ of solutions to the Yamabe problem on $M$, it is crucial to have a good understanding of the spectra of their Laplacians $\Delta_t$. In all our applications, the family $g_t$ is obtained as a deformation of the original metric $g$ on the total space $M$ of a Riemannian submersion with totally geodesic fibers \eqref{eq:submersion}, by multiplying it by a factor $t^2$ in the direction of the fibers. We will be particularly interested in the behavior of the spectrum of $\Delta_t$ as the fibers collapse to a point, i.e., as $t\to 0$.

The effect of such deformation on the spectrum was first studied by B\'erard-Bergery and Bourguignon \cite{bbb}, where $g_t$ is called \emph{canonical variation} of $g$. The starting point is that \eqref{eq:submersion} remains a Riemannian submersion with totally geodesic fibers when $g$ is replaced by $g_t$, see \cite[Prop 5.2]{bbb}. For the readers' convenience, we briefly recall some related results, that are discussed in more detail in \cite[Sec 3]{bp}.

\subsection{Vertical Laplacian and lifts of eigenfunctions}
Define the \emph{vertical Laplacian} $\Delta_v$ on a function $\psi\colon M\to\R$ by
\begin{equation*}
(\Delta_v \psi)(p):=(\Delta_F \psi|_{F_p})(p),\quad \mbox{ for all } p\in M,
\end{equation*}
where $\Delta_F$ is the Laplacian of the fiber and $F_p=\pi^{-1}(\pi(p))$ is the fiber through $p\in M$.
Just as a usual Laplacian, the vertical Laplacian is a non-negative self-adjoint unbounded operator on $L^2(M)$, however it is \emph{not elliptic} (unless $\pi$ is a covering). Since the fibers are isometric, $\Delta_v$ has discrete spectrum equal to that of the Laplacian $\Delta_F$ of the fiber. Let us denote
\begin{equation}
\begin{aligned}\label{spectra}
\spec(\Delta_M) &= \big\{0=\mu_0<\mu_1<\cdots<\mu_k\nearrow +\infty\big\},\\
\spec(\Delta_v) &= \big\{0=\phi_0<\phi_1<\cdots<\phi_j\nearrow +\infty\big\},
\end{aligned}
\end{equation}
where these eigenvalues are \emph{not repeated} according to their multiplicity. Note that the multiplicity of the eigenvalues of $\Delta_M$ is always \emph{finite}, but the eigenvalues of $\Delta_v$ might have \emph{infinite} multiplicity. For instance, $\Delta_v \widetilde\psi=0$ only implies that $\widetilde\psi$ is constant \emph{along the fibers}, i.e., $\widetilde\psi=\psi\circ\pi$ for some function $\psi\colon B\to\R$ on the base.

It is easy to see that, for any $\psi\colon B\to\R$ and its \emph{lift} $\widetilde\psi:=\psi\circ\pi$,
\begin{equation}\label{eq:laplasubm}
\Delta_M\widetilde\psi=(\Delta_B\psi)\circ\pi+g\big(\!\grad_g\widetilde\psi, \vec H\big),
\end{equation}
where $\vec H$ is the mean curvature vector field of the fibers. Since we assumed the fibers of $\pi$ are totally geodesic, $\vec H$ vanishes identically. Thus, if $\psi$ is an eigenfunction of $\Delta_B$, then its lift $\widetilde\psi$ is an eigenfunction of $\Delta_M$ with the same eigenvalue (and constant along the fibers). Therefore, there is a natural inclusion
\begin{equation}\label{eq:specincl}
\spec(\Delta_B)\subset\spec(\Delta_M).
\end{equation}
Conversely, if $\psi\colon M\to\R$ is constant along the fibers and satisfies $\Delta_M\psi=\lambda\psi$, then there exists $\check\psi\colon B\to\R$ such that $\Delta_B\check\psi=\lambda\check\psi$ and $\psi=\check\psi\circ\pi$. Summing up, it follows from \eqref{eq:laplasubm}, after checking the adequate regularity hypotheses (see \cite[Lemma 3.11]{bmp}), that the following holds.

\begin{proposition}\label{prop:eigenlift}
If $\pi\colon M\to B$ is a Riemannian submersion with totally geodesic fibers, then an eigenfunction of $\Delta_M$ is constant along the fibers if and only if it is the lift of an eigenfunction of $\Delta_B$.
\end{proposition}

\subsection{Spectrum of deformed metrics}
Another consequence of having totally geodesic fibers is that $L^2(M)$ has a Hilbert basis consisting of simultaneous eigenfunctions of the original Laplacian $\Delta_M$ and the vertical Laplacian $\Delta_v$, see \cite[Thm 3.6]{bbb}. This means that these operators can be \emph{simultaneously diagonalized}, in the appropriate sense. Furthermore, it is a simple calculation that the Laplacian $\Delta_t$ of the deformed metric $g_t$ is $\Delta_t=\Delta_M+\left(\tfrac{1}{t^2}-1\right)\Delta_v$, see \cite[Prop 5.3]{bbb}. From this, we get the following description of its spectrum.

\begin{proposition}\label{cor:specincl}
For each $t>0$, the following inclusion holds
\begin{equation*}
\spec(\Delta_t)\subset\spec(\Delta_M)+\left(\tfrac{1}{t^2}-1\right)\spec(\Delta_v).
\end{equation*}
Since the above sets are discrete, every eigenvalue $\lambda(t)$ of $\Delta_t$ is of the form
\begin{equation}\label{eq:lambdakj}
\lambda^{k,j}(t)=\mu_k+\left(\tfrac{1}{t^2}-1\right)\phi_j,
\end{equation}
for some eigenvalues $\mu_k$ and $\phi_j$ of $\Delta_M$ and $\Delta_v$, respectively.
\end{proposition}

\begin{proof}
See \cite[Cor 3.6]{bp}.
\end{proof}

\begin{remark}
Not all possible combinations of $\mu_k$ and $\phi_j$ in \eqref{eq:lambdakj} give rise to an eigenvalue of $\Delta_t$. In fact, this only happens when the total space of the submersion is a Riemannian product. In general, determining which combinations are allowed depends on the global geometry of the submersion.
\end{remark}

Note that, since the fibers of $\pi$ remain totally geodesic with respect to $g_t$, \eqref{eq:specincl} holds when $\Delta_M$ is replaced with $\Delta_t$, i.e.,
\begin{equation}\label{eq:specincl2}
\spec(\Delta_B)\subset\spec(\Delta_t), \quad\mbox{ for all } t>0.
\end{equation}
Moreover, when $j=0$ in \eqref{eq:lambdakj}, $\lambda^{k,0}(t)=\mu_k\in\spec(\Delta_M)$ remains an eigenvalue of $\Delta_t$ for $t\neq1$ if and only if $\mu_k\in\spec(\Delta_B)$. Such eigenvalues $\lambda^{k,0}(t)$ of $\Delta_t$ will be called \emph{constant eigenvalues}, since they are independent of $t$. In other words, the constant eigenvalues of $\Delta_t$ are the ones in the left hand side of \eqref{eq:specincl2}. We stress that $\lambda^{k,0}(t)$ is not necessarily a constant eigenvalue for all $k$.


\section{Bifurcation on homogeneous fibrations}\label{sec:bifhom}
Let $H\subsetneq K\subsetneq G$ be compact connected Lie groups, such that $\dim K/H\geq 2$, and assume additionally that
either $H$ is normal in $K$, or $K$ is normal in $G$. Consider the homogeneous fibration \eqref{eq:homfib},
\begin{equation*}
K/H\longrightarrow G/H\stackrel{\pi}{\longrightarrow} G/K, \quad \pi(gH)=gK,
\end{equation*}
and notice that the fiber over $gK \in G/K$ is $(gK)H\subset G/H$. Define a $K$-action on $G/H$ by $k\cdot gH=kgH$ if $K$ is normal in $G$ and by $k\cdot gH=gk^{-1}H$ if $H$ is normal in $K$. Notice that the orbits of this $K$-action are exactly the fibers of $\pi$.
Denote by $\mathfrak h\subsetneq\mathfrak k\subsetneq\mathfrak g$ the Lie algebras of $H\subsetneq K\subsetneq G$. We henceforth fix an $\Ad_G(K)$-invariant complement $\mathfrak m$ to $\mathfrak k$ in $\mathfrak g$, and an $\Ad_G(H)$-invariant complement $\mathfrak p$ to $\mathfrak h$ in $\mathfrak k$, i.e.,
\begin{equation*}
\mathfrak k\oplus\mathfrak m=\mathfrak g, \quad [\mathfrak k,\mathfrak m]\subset\mathfrak m \quad \mbox{ and } \quad \mathfrak h\oplus\mathfrak p=\mathfrak k, \quad [\mathfrak h,\mathfrak p]\subset\mathfrak p.
\end{equation*}
There are natural identifications of $\mathfrak m$ and $\mathfrak p$ with the tangent spaces to $G/K$ and $K/H$ at the origin,\footnote{i.e., $\mathfrak m\cong T_{(eK)}G/K$ and $\mathfrak p\cong T_{(eH)}K/H$.} respectively. The sum $\mathfrak m\oplus\mathfrak p$ is an $\Ad_G(H)$-invariant complement to $\mathfrak h$ in $\mathfrak g$, which is identified with the tangent space to $G/H$ at the origin.

Any $\Ad_G(K)$-invariant inner product on $\mathfrak m$ defines a $G$-invariant metric on the base $G/K$; and any $\Ad_G(H)$-invariant inner product on $\mathfrak p$ defines a $K$-invariant metric on the fiber $K/H$. The orthogonal direct sum of these inner products on $\mathfrak m\oplus\mathfrak p$ now gives a $G$-invariant metric on $G/H$. We call these metrics on \eqref{eq:homfib} \emph{compatible} homogeneous metrics. We stress that not necessarily all $G$-invariant metrics on $G/H$ arise in this way. However, if $\mathfrak m$ and $\mathfrak p$ do not share equivalent $\Ad_G(H)$-submodules, then all $G$-invariant metrics on $G/H$ that project to a $G$-invariant metric on $G/K$ are of this form. We will henceforth consider all homogeneous fibrations to be endowed with compatible metrics. The homogeneous fibration $\pi\colon G/H\to G/K$ is then automatically a Riemannian submersion with totally geodesic fibers (isometric to $K/H$), see \cite[Prop 2]{bb} or \cite[Thm 9.80]{besse}.

\begin{theorem}\label{thm:bifhom}
Let $K/H\to G/H\to G/K$ be a homogeneous fibration as above, and assume that $K/H$ has positive scalar curvature. Let $g_t$ be the family of homogeneous metrics on $G/H$ obtained by scaling the fibers by $t^2$. Then, there exists a sequence $\{t_q\}$ in $]0,1[$, that converges to $0$, of bifurcation values for the family $g_t$.
\end{theorem}

\begin{proof}
Since $\pi\colon(G/H,g_t)\to G/K$ is a Riemannian submersion with totally geodesic fibers,
\begin{equation}\label{eq:scal}
\scal(G/H,g_t)=\tfrac{1}{t^2}\scal(K/H)+\scal(G/K)\circ\pi-t^2\|A\|^2,
\end{equation}
where $\|A\|$ is the Hilbert-Schmidt norm of the Gray-O'Neill tensor of integrability of the horizontal distribution (see~\cite[Prop 9.70]{besse}). Note that $\scal(K/H)>0$ implies
\begin{equation}\label{eq:limit}
\lim_{t\to0_{+}}{\scal(G/H,g_t)}=+\infty.
\end{equation}

Recall that the \emph{degeneracy values} in this setup are those $t>0$ such that
\begin{equation}
(m-1)\Delta_{t} \psi -\scal(G/H,g_{t})\psi=0
\end{equation}
has a non-trivial solution $\psi$, where $m=\dim G/H$ and $\Delta_t$ is the Laplacian of $(G/H,g_t)$. From \eqref{eq:lambdakj} and \eqref{eq:scal}, the set of such degeneracy values is discrete. From \eqref{eq:limit}, there are infinitely many degeneracy values $t_q$ accumulating at $0$ such that $\scal(G/H,g_{t_q})/(m-1)\in\spec(\Delta_{G/K})\subset\spec(\Delta_t)$, see \eqref{eq:specincl2}. We claim that every such value $t_q$ is a bifurcation value.

Fix one such $t_q$, and denote by $\lambda\in\spec(\Delta_{G/K})$ the constant eigenvalue of $\Delta_{t}$ such that $\scal(G/H,g_{t_q})/(m-1)=\lambda$. If there is a change in the Morse index at $t_q$, i.e., for $\varepsilon>0$ sufficiently small, $N(g_{t_q-\varepsilon})\neq N(g_{t_q+\varepsilon})$, then by Proposition~\ref{prop:bifmorseindex}, $t_q$ is a bifurcation value. However, if the Morse index does not change, there must be a \emph{compensation} of eigenvalues. Namely, there must exist non-constant eigenvalues $\lambda^{k_1,j_1}(t),\dots,\lambda^{k_n,j_n}(t)$ of $\Delta_t$, whose combined multiplicity equals the multiplicity of $\lambda$, such that
\begin{equation*}
\begin{aligned}
\lambda<\scal(G/H,g_t)/(m-1)<\lambda^{k_i,j_i}(t), \quad & \mbox{ for all } t<t_q \mbox{ (close to } t_q) \mbox{ and } 1\leq i\leq n, \\
\lambda>\scal(G/H,g_t)/(m-1)>\lambda^{k_i,j_i}(t), \quad & \mbox{ for all } t>t_q \mbox{ (close to } t_q) \mbox{ and } 1\leq i\leq n. \\
\end{aligned}
\end{equation*}
Denoting by $E_t^\alpha$ the eigenspace of the eigenvalue $\alpha\in\spec(\Delta_t)$, we have the negative isotropic representations $\pi^-_t$ on the linear spaces (of same finite dimension):
\begin{equation}\label{eq:negspaces}
\begin{aligned}
&E\oplus E^\lambda_t,  &\mbox{ for } t<t_q \mbox{ (close to } t_q), \\
&E\oplus \bigoplus_i E^{\lambda^{k_i,j_i}}_t,   &\mbox{ for } t>t_q \mbox{ (close to } t_q), \\
\end{aligned}
\end{equation}
where $E$ is the space spanned by the eigenfunctions with eigenvalues less than $\scal(G/H,g_t)/(m-1)$ for $t$ close to $t_q$. We claim that for small $\varepsilon>0$, the negative isotropic representations $\pi_{t_q-\varepsilon}^-$ and $\pi_{t_q+\varepsilon}^-$ on the spaces \eqref{eq:negspaces} \emph{cannot be isomorphic}. From Proposition~\ref{prop:equivbif}, it then follows that $t_q$ is a bifurcation value, concluding the proof.

Let us verify the above claim. Given a representation $\pi$ of a compact group, denote by $\mathfrak I(\pi)$ the number of copies of the trivial representation in the irreducible decomposition of $\pi$. It is easily seen that a necessary condition for the two representations $\pi_a$ and $\pi_b$ to be isomorphic is that $\mathfrak I(\pi_a)=\mathfrak I(\pi_b)$. For the negative isotropic representation $\pi_t^-$, one can compute
\begin{equation}\label{eq:auxidentity}
\mathfrak I(\pi_t^-)=\mathop{\sum_{\eta\in\spec(\Delta_{G/K})}}_{\eta<\frac{\scal(g_{t})}{m-1}} \mathrm{mul}(\eta),
\end{equation}
where $\mathrm{mul}(\eta)$ is the multiplicity of $\eta$ as an eigenvalue of $\Delta_{G/K}$. Indeed, an eigenfunction $\psi$ of $\Delta_t$ is constant along the fibers $K/H$ of the homogeneous fibration \eqref{eq:homfib} if and only if it is $K$-invariant, i.e., $\psi\circ k=\psi$ for all $k\in K$. This is equivalent to saying that $\psi$ is a fixed point of $\pi_t^-$, see \eqref{eq:repr}. So, from Proposition~\ref{prop:eigenlift}, the left hand side of \eqref{eq:auxidentity} is greater than or equal to the right hand side. Conversely, it is easy to see that if $\psi\colon G/H\to\R$ is the linear combination of eigenfunctions $\psi_i\colon G/H\to\R$ of $\Delta_t$, and if $\psi$ is constant along the fibers of $G/H\to G/K$, then each $\psi_i$ must be constant along such fibers. This follows from the fact that the subspace of $L^2(G/H)$ of functions that are constant along the fibers (which is isomorphic to $L^2(G/K)$), is spanned by the (lift of) eigenfunctions of $\Delta_{G/K}$. In other words, the space spanned by the eigenfunctions of $\Delta_{G/H}$ that are constant along the fibers and the space spanned by the eigenfunctions of $\Delta_{G/H}$ that are not constant along the fibers are $L^2$-orthogonal. This completes the proof of \eqref{eq:auxidentity}.

From \eqref{eq:negspaces} and \eqref{eq:auxidentity} we see that, for any $\varepsilon>0$ small, $\mathfrak I(\pi_{t_q-\varepsilon}^-)>\mathfrak I(\pi_{t_q+\varepsilon}^-)$. Therefore these representations are not isomorphic, concluding the proof.
\end{proof}

\begin{remark}
At all bifurcation values for the family $g_t$ a \emph{break of symmetry} occurs, in the sense that any solutions in the bifurcating branch are not $G$-homogeneous. This follows easily from the fact that each conformal class contains at most one homogeneous metric (up to rescaling).
\end{remark}

\subsection{Sharpness of fiber hypotheses}
If the fibers $K/H$ have flat scalar curvature or have dimension $1$, then $\scal(g_t)$ remains bounded as $t\to 0$, and there are not infinitely many degeneracy values as above. For instance, consider a fibration of tori, $G=K\times K$, $K=T^2$ and $H=\{e\}$, where the inclusion of $K$ is as one of the factors of $G$. If $K$ is endowed with the flat metric and $G$ with the product metric, shrinking the fibers keeps the total space $G/H$ flat, hence the family obtained is (trivially) locally rigid, for all $t>0$.

\section{Examples}\label{sec:metrics}

We now discuss how to construct examples of homogeneous fibrations with fibers of positive scalar curvature to which Theorem~\ref{thm:bifhom} (hence also Theorem~\ref{thm:main}) applies.

\subsection{Normal homogeneous metrics}
A $K$-invariant metric on $K/H$ is called \emph{normal} if it is obtained from the restriction to $\mathfrak p$ of a bi-invariant inner product on $\mathfrak k$. Since $K$ is compact, it admits a bi-invariant metric. Hence, normal homogeneous metrics always exist\footnote{There exist homogeneous spaces on which \emph{all} homogeneous metrics are normal. These spaces are called of \emph{normal type} in \cite[\S 7]{bb}. Examples are spaces whose isotropy representation is irreducible (e.g., irreducible symmetric spaces) and their products.} on $K/H$. Endowed with such a metric, the sectional curvature of a tangent plane at the origin, spanned by orthonormal vectors $X$ and $Y$, is
\begin{equation}
\sec(X,Y)=\tfrac14\big\|\big[\overline X,\overline Y\big]\big\|^2+\tfrac34\big\|\big[\overline X,\overline Y\big]_\mathfrak h\big\|^2\geq0,
\end{equation}
where $\overline X=(0,X)$ and $\overline Y=(0,Y)$ are the horizontal lifts of $X$ and $Y$ to $\mathfrak k=\mathfrak h\oplus\mathfrak p$ and $(\cdot)_\mathfrak h$ denotes the $\mathfrak h$-component of a vector in $\mathfrak k$. In particular, $\scal(K/H)\geq 0$ and it is equal to zero if and only if $\mathfrak p$ is an abelian ideal of $\mathfrak k$. This, in turn, is equivalent to the existence of an abelian subgroup $A\subset K$ that acts transitively on $K/H$. Since (the closure of) $A$ is a compact connected abelian Lie group, it must be a torus. Hence, $K/H$ itself is a torus. Thus, \emph{any} normal metric on $K/H$ has positive scalar curvature, unless $K/H$ is a torus.

More generally, $K/H$ admits a (normal) metric of positive scalar curvature if and only if its universal covering is not diffeomorphic to an Euclidean space, see \cite[Thm 2]{bb}. In particular, if $K/H$ is not an aspherical manifold, then $K/H$ has a metric of positive scalar curvature. We also stress that, generally, there are other $K$-invariant metrics (not necessarily normal) on $K/H$ that have positive scalar curvature, and any such metric can be used to obtain examples of applications of our results. In this direction, examples with non-normal homogeneous metrics with positive scalar curvature on spheres will be discussed below.

In any of the above cases, one can endow the remaining spaces of \eqref{eq:homfib} with compatible homogeneous metrics. In this way, it is possible to construct many classes of homogeneous fibrations to which our results apply.

\subsection{Twisted products}\label{sec:twist}
Let us now describe explicit triples $H\subsetneq K\subsetneq G$ of compact connected Lie groups with either $H$ normal in $K$ or $K$ normal in $G$. Starting with the latter, if a compact connected Lie group $G$ has a proper connected normal subgroup $K$, then there exists another connected normal subgroup\footnote{The subgroup $L$ is obtained as the connected subgroup of $G$ whose Lie algebra $\mathfrak l$ is the orthogonal complement (with respect to a bi-invariant metric) of the Lie algebra $\mathfrak k$ of $K$. Since $K$ is normal, $\mathfrak k$ is an ideal and then $\mathfrak l=\mathfrak k^\perp$ is also an ideal, because the adjoint representation is skew-symmetric with respect to the bi-invariant metric. This implies that $L$ is normal, and $KL$ generates the entire group $G$ by connectedness, since it does so locally near the identity. Finiteness of $\Gamma=K\cap L$ follows since the intersection $\mathfrak k\cap\mathfrak l$ is trivial.} $L$ of $G$ such that $G=(K\times L)/\Gamma$, where $\Gamma\subset K\times L$ is finite. For any subgroup $H$ of $K$, one then gets the homogeneous fibration
\begin{equation}\label{eq:twistedprod}
K/H\longrightarrow ((K\times L)/\Gamma)/H \longrightarrow G/K.
\end{equation}
This provides an algorithm to build examples, whose input are the groups $H$, $K$, $L$ and $\Gamma$. Setting $G=(K\times L)/\Gamma$ we then have that the factor $K$ is a normal subgroup.

\begin{example}\label{ex:so4}
For instance, consider $G=\SO(4)$, which is double-covered by $S^3\times S^3$. In this case, $K=L=S^3$ and $\Gamma=\Z_2$ is the diagonal embedding into $K\times L$, i.e., $\Gamma$ is the subgroup generated by $(-1,-1)\in S^3\times S^3$. We then have $G=(K\times L)/\Gamma=(S^3\times S^3)/\Z_2$, and the quotient $G/K$ is isomorphic to $S^3/\Z_2=\SO(3)$. One can choose $H\subset K=S^3$ to be trivial, so that $K/H=S^3$; or to be a circle, e.g., the circle that gives the Hopf action on $S^3$, so that $K/H=S^2$. The corresponding homogeneous fibrations that \eqref{eq:twistedprod} gives are:
\begin{equation*}
S^3\to \SO(4)\to \SO(3) \quad\mbox{ and } \quad S^2\to \SO(4)/S^1\to \SO(3).
\end{equation*}
\end{example}

In the above cases, the total space $G/H$ is a \emph{twisted product}; while in the special case where $\Gamma$ is trivial, we have the splitting $G=K\times L$, so that $G/K=L$ and $G/H=K/H\times L$ is an actual \emph{product} manifold, i.e., \eqref{eq:twistedprod} becomes
\begin{equation}\label{eq:product}
K/H\longrightarrow K/H\times L\longrightarrow L.
\end{equation}
Differently from the above cases, the deformed metrics $g_t$ in this situation are product metrics, obtained by rescaling the directions tangent to the first factor $K/H$ by $t^2$ and keeping the metric constant in the directions tangent to $L$. In this way, any product of a compact homogeneous space (with positive scalar curvature) and a compact connected Lie group satisfies the hypotheses of Theorem~\ref{thm:bifhom} (hence also of Theorem~\ref{thm:main}). We note that this particular case is covered by the results of \cite{lpz}.

\subsection{Sphere fibers}
An important observation is that in the above product situation where $K/H=S^n$ is a sphere, our result allows for \emph{any} homogeneous metric (not necessarily normal) on $K/H=S^n$ whose scalar curvature is positive, as opposed to only the round metric. Homogeneous metrics on spheres were classified by Ziller~\cite{ziller} in 1982. They are obtained by rescaling the fibers of one of the Hopf fibrations:
\begin{equation*}
S^1\to S^{2n+1}\to\C P^n, \quad S^3\to S^{4n+3}\to\Hr P^n, \quad S^7\to S^{15}\to S^8\left(\tfrac12\right).
\end{equation*}
In the first and last case, there is only one direction in which the fibers can be rescaled, while for the fibration with $S^3$ fiber, each of the $3$ coordinate $S^1$ subgroups can be rescaled with a different factor. This gives rise to the metrics:
\begin{itemize}
\item $\g_s$, a $1$-parameter family of $\U(n+1)$-invariant metrics on $S^{2n+1}$;
\item $\h_{s_1,s_2,s_3}$, a $3$-parameter family of $\Sp(n+1)$-invariant metrics on $S^{4n+3}$;
\item $\k_s$, a $1$-parameter family of $\Spin(9)$-invariant metrics on $S^{15}$.
\end{itemize}
All the above metrics have $\scal>0$ for a certain range of parameters (see table below). Thus, if $K/H=S^n$ is isometric to one of those spheres, the submersion \eqref{eq:product} satisfies the hypotheses of Theorem~\ref{thm:bifhom} (hence also of Theorem~\ref{thm:main}).

\begin{center}
\begin{tabular*}{\textwidth}{@{\extracolsep{\fill}} l l l l}
\hline
metric & $K$ \rule[-1.2ex]{0pt}{0pt} \rule{0pt}{2.5ex} & $H$ & $\dim K/H$ \\
\hline
\hline
\multirow{3}{*}{round} \rule{0pt}{2.5ex} & $\SO(n+1)$ & $\SO(n)$ & $n\, (\geq 2)$ \\
& $\Spin(7)$ & $G_2$ & $7$ \\
& $G_2$ \rule[-1.2ex]{0pt}{0pt}  & $\SU(3)$ & $6$ \\
\hline
\multirow{2}{*}{$\g_s$} \rule{0pt}{2.5ex} & $\SU(n+1)$ & $\SU(n)$ & $2n+1\, (\geq 2)$ \\
& $\U(n+1)$ & $\U(n)$ \rule[-1.2ex]{0pt}{0pt}  & $2n+1\, (\geq 2)$ \\
\hline
\multirow{3}{*}{$\h_{s_1,s_2,s_3}$} \rule{0pt}{2.5ex} & $\Sp(n+1)$ & $\Sp(n)$ & $4n+3$ \\
& $\Sp(n+1)\times\Sp(1) $ & $\Sp(n)\times\Sp(1)$ & $4n+3$ \\
& $\Sp(n+1)\times\U(1)$ \rule[-1.2ex]{0pt}{0pt}  & $\Sp(n)\times\U(1)$ & $4n+3$ \\
\hline
$\k_s$ \rule[-1.2ex]{0pt}{0pt} & $\Spin(9)$ & $\Spin(7)$ & $15$ \\
\hline
\end{tabular*}
\end{center}
\smallskip
The range of parameters for which the above $K$-invariant metrics have $\scal>0$ is $0<s<s_{max}$, where (see \cite[Prop 4.2]{bp}):
\begin{center}
\begin{tabular*}{\textwidth}{@{\extracolsep{\fill}} lccc}
\hline
& $\g_s$ \rule[-1.2ex]{0pt}{0pt} & $\h_{s,s,s}$ & $\k_s$ \\
\hline
\hline
$s_{max}$ \rule{0pt}{4.0ex} \rule[-2ex]{0pt}{0pt}  & $\sqrt{2n+2}$ & $\sqrt{\tfrac{2n+4}{3}+\tfrac{\sqrt{18n+16(n^2+2n)^2}}{6n}}$ & $\sqrt{2+\tfrac{\sqrt{19}}{2}}$ \\
\hline
\end{tabular*}
\end{center}
When $s=1$, the above metrics are isometric to the round metric (which is the only normal homogeneous metric in each family).

\begin{remark}
The above construction can be interpreted as having a chain $$H\subsetneq H'\subsetneq K\subsetneq G$$ of Lie groups, and first performing a Cheeger deformation with respect to the $H'$-action in the total space of $H'/H\to K/H\to K/H'$ to obtain positive scalar curvature on $K/H$; and then using this metric on the fiber of $K/H\to G/H\to G/K$. The Cheeger deformation with respect to the $K$-action on $G/H$ gives the desired $1$-parameter family $g_t$ that satisfies the hypotheses of our results. More generally, one could perform multiple ``preliminary'' Cheeger deformations with a longer chain of groups in order to gain $\scal>0$ on the fibers of the ``last'' homogeneous fibration.
\end{remark}

\begin{example}
Another interesting class of examples with $K$ normal in $G$ is when $H$ is trivial, so that the resulting homogeneous fibration \eqref{eq:homfib} is a short exact sequence of Lie groups $K\to G\to G/K$. Note this is precisely the case of $S^3\to \SO(4)\to \SO(3)$ in Example~\ref{ex:so4}. Here, the deformed metrics $g_t$ are obtained by shrinking the original metric in the direction of the cosets of $K$ in $G$.
\end{example}

\subsection{Other examples}
As explained above, instead of having $K$ normal in $G$, one can also consider the case where $H$ is normal in $K$. This poses far less restrictions on the homogeneous fibrations that can be obtained, since the group $K$ will split as a product (up to a finite quotient), however no conditions are imposed on $G$. For instance, $G$ may have arbitrarily large dimension and rank. It follows from the above discussion on normal homogenous metrics that our results apply to the submersion \eqref{eq:homfib} with $H$ normal in $K$ as soon as the quotient $K/H$ is not abelian. More precisely, if $K/H$ is not a torus, then any normal homogeneous metric will have positive scalar curvature. Any choice of $G$ will then yield a triple $H\subsetneq K\subsetneq G$ whose corresponding homogeneous fibration can be endowed with compatible metrics for which Theorem~\ref{thm:bifhom} (hence also Theorem~\ref{thm:main}) applies.

\begin{example}
To illustrate the above comments, let us build on the case described in Example~\ref{ex:so4}. Instead of $G$, set $K=\SO(4)=(S^3\times S^3)/\Z_2$ and $H$ as one of the $S^3$ factors, so that $K/H=\SO(3)$. Then $G$ can be chosen arbitrarily among compact connected Lie groups that have a subgroup isomorphic to $\SO(4)$. As concrete examples, we may set $G=\SO(5)$ so that $G/K=S^4$ is a sphere; or $G=\SO(6)$ so that $G/K=T_1 S^4$ is the unit tangent bundle of $S^4$. The corresponding homogeneous fibrations are
\begin{equation*}
\SO(3)\to \SO(5)/S^3 \to S^4 \quad\mbox{ and } \quad \SO(3)\to \SO(6)/S^3\to T_1 S^4.
\end{equation*}
\end{example}

\begin{remark}
Any of the above examples can be trivially used to obtain new ones with non-simply-connected total space. Consider $F\subset K$ a finite subgroup and its action on $K/H$ and $G/H$, so that the inclusion map $K/H\to G/H$ is equivariant. One can form a new fibration replacing $K/H$ and $G/H$ by their (non-simply-connected) quotients by the $F$-action. Since $F\subset K$, the base of the fibration remains $G/K$. If the original metrics satisfied the conditions of Theorem~\ref{thm:bifhom}, then the induced metrics in the new fibration also do.
\end{remark}


\section{Bifurcation on non-homogeneous fibrations}\label{sec:curv}

A natural question is how the presence of many symmetries affects the bifurcation result obtained above. Homogeneity played a pivotal role in employing the equivariant bifurcation criterion (Proposition~\ref{prop:equivbif}). When this assumption is dropped, the only tool at hand is the Morse index criterion (Proposition~\ref{prop:bifmorseindex}), so extra hypotheses are needed to guarantee a change in the Morse index at the degeneracy values. One such possibility is to impose certain curvature conditions that allow to bound (from below) the growth of the eigenvalues of a non-homogeneous collapsing Riemannian submersion.

\begin{theorem}\label{thm:bifcurv}
Let $F\to M\to B$ be a Riemannian submersion with totally geodesic fibers and  $l=\dim F\geq 2$, $m=\dim M$. Assume the metrics $g_t$ obtained by shrinking the fibers have constant scalar curvature, and that for some $\tau>0$ and $k_1,k_2>0$,
\begin{equation*}
\left\{\begin{matrix}
\Ric_F \!\!\!\!& \geq &\!\!\!\! (l-1)\,k_1 \\
\scal_F \!\!\!\!&<&\!\!\!\! l(m-1)\,k_1
\end{matrix} \right.
\quad\quad\mbox{and}\quad\quad
\left\{\begin{matrix}
\Ric_{(M,g_{\tau})} \!\!\!\!&\geq & \!\!\!\!(m-1)\,k_2 \\
\scal_B \!\!\!\!&\leq &\!\!\!\! m(m-1)\,k_2.
\end{matrix} \right.
\end{equation*}
Then, there exists a sequence $\{t_q\}$ in $]0,\tau[$, that converges to $0$, of bifurcation values for the family $g_t$.
\end{theorem}

\begin{proof}
Since $\pi\colon M\to F$ is a Riemannian submersion with totally geodesic fibers,
\begin{equation}\label{eq:scalt}
\scal(M,g_t)=\tfrac{1}{t^2}\scal_F+\scal_B\circ\pi-t^2\|A\|^2,
\end{equation}
where $\|A\|$ is the Hilbert-Schmidt norm of the Gray-O'Neill tensor $A$. From $\scal_F>0$, we have $\lim_{t\to0_{+}}\scal(M,g_t)=+\infty$. As before, the set of degeneracy values is discrete; and infinitely many of them occur due to $\spec(\Delta_B)\subset\spec(\Delta_t)$, see \eqref{eq:specincl2}. Denote by $t_q$ the sequence of degeneracy values, accumulating at $0$, such that $\scal(M,g_{t_q})/(m-1)\in\spec(\Delta_B)$. We claim that for $q$ sufficiently large (i.e., $t_q$ sufficiently small), $t_q$ is a bifurcation value.

From Proposition~\ref{prop:bifmorseindex}, we must verify that, for $t_q$ sufficiently small, there is a change of the Morse index $N(g_t)$ at $t_q$. It suffices to prove that every \emph{non-constant} eigenvalue $\lambda^{k,j}(t)$ of $\Delta_t$ is strictly larger than $\scal(M,g_t)/(m-1)$ for $t$ sufficiently small, so that no compensation of eigenvalues can occur (cf. proof of Theorem~\ref{thm:bifhom}). Up to a simple rescaling, assume $\tau=1$. Since the eigenvalues $\mu_k$ of $\Delta_M$ and $\phi_j$ of $\Delta_v$ are ordered to be monotonically increasing, it suffices to prove
\begin{equation}\label{eq:goal}
\scal(M,g_t)/(m-1)< \lambda^{1,1}(t)=\mu_1+\left(\tfrac{1}{t^2}-1\right)\phi_1, \quad\mbox{for } t \mbox{ sufficiently small},
\end{equation}
see Corollary~\ref{cor:specincl}. From the Lichnerowicz estimates, since $\Ric_F\geq(l-1)\,k_1$ and $\Ric_M\geq(m-1)\,k_2$, we have:
\begin{equation}\label{eq:lichnerowicz}
\phi_1\geq l\,k_1 \quad \mbox{ and }\quad \mu_1\geq m\,k_2,
\end{equation}
see \cite[Chap 3, Thm 9]{chavel}. Combining the latter with $\scal_B\leq m(m-1)\,k_2$, we get
\begin{equation}\label{eq:1}
\scal_B\circ\pi-t^2\|A\|^2\leq \scal_B\circ\pi \leq (m-1)\mu_1.
\end{equation}
Also, from \eqref{eq:lichnerowicz} and $\scal_F<l(m-1)\, k_1$, it follows that $\scal_F<(m-1)\phi_1$. Thus, for $t$ sufficiently small, we have $\scal_F<(1-t^2)(m-1)\phi_1$, hence
\begin{equation}\label{eq:2}
\tfrac{1}{t^2}\scal_F<(m-1)\left(\tfrac{1}{t^2}-1\right)\phi_1.
\end{equation}
Adding \eqref{eq:1} and \eqref{eq:2} and using \eqref{eq:scalt}, we obtain \eqref{eq:goal}, concluding the proof.
\end{proof}

The above result can be applied, e.g., to low dimensional Hopf fibrations, reobtaining the conclusion of Theorem~\ref{thm:bifhom}. Nevertheless, for larger dimensions, the curvature pinching conditions are not satisfied, although the result remains true.

\begin{remark}\label{rem:improvements}
The curvature pinching conditions above are solely needed to avoid compensation of eigenvalues, in a rather forceful way. Given a Riemannian submersion $F\to M\to B$ with totally geodesic fibers of positive scalar curvature, suppose the metrics $g_t$ obtained by shrinking the fibers have constant scalar curvature. If under certain conditions there is an inclusion of (a non-trivial subgroup of) the isometry group of $F$ in the isometry group of $(M,g_t)$, $t\in\,]0,\tau[$, then one can employ the equivariant techniques to deal with possible compensation of eigenvalues and still obtain infinitely many bifurcation values in this non-homogeneous context.
\end{remark}


\section{Multiplicity of solutions to the Yamabe problem}\label{sec:final}

We now explain how to obtain the multiplicity result claimed in the Introduction.

\begin{proposition}\label{prop:mult}
Let $g_t$, $t\in\,]0,\tau[$, be a family of metrics on $M$ with constant scalar curvature and $N(g_t)>0$. Suppose there exists a sequence $\{t_q\}$ in $]0,\tau[$, that converges to $0$, of bifurcation values for $g_t$. Then, there is an infinite subset $\mathcal T\subset\,]0,\tau[$, that accumulates at $0$, such that for each $t\in\mathcal T$, there are at least $3$ solutions to the Yamabe problem in the conformal class $[g_t]$.
\end{proposition}

\begin{proof}
For all $t$, denote by $\hat g_t$ the unit volume metric homothetic to $g_t$. Since $t_q$ is a bifurcation value, there are values of $t$ arbitrarily close to $t_q$ for which the conformal class $[g_t]$ contains a unit volume constant scalar curvature metric $g$ distinct from $\hat g_t$. Since $N(g_t)>0$, by continuity of the Morse index, also $N(g)>0$. In particular, neither $\hat g_t$ nor $g$ are minima of the Hilbert-Einstein functional in $[g_t]$. Therefore, $[g_t]$ contains at least 3 distinct unit volume constant scalar curvature metrics, i.e., 3 solutions to the Yamabe problem. The set $\mathcal T$ of such $t$'s clearly accumulate at $0$, since $t_q$ converges to $0$.
\end{proof}

Theorems \ref{thm:main} and \ref{thm:curv} now follow easily from Theorems~\ref{thm:bifhom} and \ref{thm:bifcurv}, respectively. Indeed, in order to apply Proposition~\ref{prop:mult}, it is just necessary to verify that $N(g_t)>0$. In the first case, if $g_t$ comes from a homogeneous fibration, there must be a value $\tau$ of $t$, when $\scal(t)/(m-1)$ crosses the first eigenvalue from the base, before any compensation is even possible. At this value $t=\tau$, the Morse index changes from $0$ to a positive integer. Then, for $t\in\,]0,\tau[$, we have $N(g_t)\geq N(g_{\tau-\varepsilon})>0$. In the second case, $N(g_t)$ gets arbitrarily large as $t\to 0$, so this condition is also satisfied.


\end{document}